\documentclass[12pt]{amsart}

\usepackage{cite}
\usepackage{ljm-auth}

\newtheorem{theorem}{Theorem}
\newtheorem{corollary}{Corollary}
\newtheorem{lemma}{Lemma}
\newtheorem{remark}{Remark}

\author{S.\,Grigoryan, T.\,Grigoryan, E.\,Lipacheva, A.\,Sitdikov}
\crauthor{Grigoryan, Grigoryan, Lipacheva, Sitdikov} 

\tit{$C^*$-algebra generated by the path semigroup}
\shorttit{$C^*$-algebra generated by the path semigroup} 

\begin{document}

\maketit

\address{Kazan State Power Engineering University, Krasnosel'skaya Str., 51, Kazan, Tatarstan, 420066, Russia}

\email{gsuren@inbox.ru, tkhorkova@gmail.com, elipacheva@gmail.com,
airat$_{-}$vm@rambler.ru}

\Received{July 4, 2016}

\abstract{In this paper we study the structure of the $C^*$-algebra,
generated by the representation of the path semigroup on a partially
ordered set (poset) and get a net of isomorphic $C^*$-algebras over
this poset. We construct the extensions of this algebra, such that
the algebra is an ideal in that extensions and quotient algebras are
isomorphic to the Cuntz algebra.} \notes{0}{

}

\section{Introduction}

In the algebraic approach to the quantum field theory
\cite{Haag}\hspace{0.2cm}(the algebraic quantum field theory) the
physical content of the theory is encoding by a collection of
$C^*$-algebras of observables $\mathcal{A}=\{\mathcal{A}_o\}_{o\in
K}$ indexed by elements of a partially ordered set $K$ (poset)
\cite{Ruzzi}. The poset $K$ is a non-empty set with a binary
relation $\leq$ which is reflexive, antisymmetric and transitive. A
net of $C^*$-algebras over the poset $K$ is the pair
$(\mathcal{A},\gamma)_{K}$, where
$\gamma=\{\gamma_{o'o}:\mathcal{A}_o\rightarrow\mathcal{A}_{o'}\}_{o\leq
o'}$ are *-morphisms fulfilling the net relations
$$\gamma_{o''o}=\gamma_{o''o'}\circ\gamma_{o'o}$$ for all $o\leq o'\leq o''\in K$.
If we consider the poset $K$ as a category in which objects are
elements of $K$ and morphisms are arrows $(o,o')$ for all $o\leq
o'\in K$, then the net of $C^*$-algebras represents a covariant
functor from a poset $K$ to category of unital $C^*$-algebras with
*-morphisms (see for example \cite{Roberts,RuzziVasselli}). More
precisely we have a net of $C^*$-algebras for an upward directed
poset and in the event of non-upward directed we obtain a precosheaf
of $C^*$-algebras \cite{Vasselli1,Vasselli2,BrunettiRuzzi}.

In this paper we give an algebraic notion of a path on a poset $K$
which turns out to be relevant to the point of view on a path as a
sequence of 1-simplices. We introduce the path semigroup $S$ on the
given poset $K$ and construct a new $C^*$-algebra $C^*_{red}(S)$
generated by the representation of $S$. We consider both an upward
directed set $K$ and non-upward directed. The present paper is
addressed to detailed study of the path semigroup $S$ and the
$C^*$-algebra $C^*_{red}(S)$. We construct the net of isomorphic
$C^*$-algebras $\left\{\mathcal{A}_a,\gamma_{ba},a\leq b
\right\}_{a,b\in K}$ over the poset $K$, where $\mathcal{A}_a$ are
restrictions of the algebra $C^*_{red}(S)$ on Hilbert subspaces and
$\gamma_{ba}:\mathcal{A}_a \rightarrow \mathcal{A}_b$ are
*-isomorphisms, such that $\gamma_{cb}\circ\gamma_{ba}=\gamma_{ca}$
for all $a\leq b\leq c\in K$. In the last section we consider
extensions $C^*_{red,n}(S)$ and $C^*_{red,\infty}(S)$ of the algebra
$C^*_{red}(S)$. We prove that $C^*_{red}(S)$ is an ideal in
$C^*_{red,n}(S)$ and also in $C^*_{red,\infty}(S)$. We show that
quotient algebras $C^*_{red,n}(S)/C^*_{red}(S)$ and
$C^*_{red,\infty}(S)/C^*_{red}(S)$ are isomorphic to the Cuntz
algebra.

Several works in recent years have addressed the $C^*$-algebras
generated by the left regular representations of semigroups with
reduction \cite{AGL} and by the representations of an inverse
semigroup \cite{AT,GT,GLT}. In the paper \cite{ANS} have shown that
the Cuntz algebra can be represented as a $C^*$-crossed product by
endomorphisms of the CAR algebra.

\section{Path semigroup}

In this section we define the path semigroup $S$ on a partially
ordered set $K$. The semigroup $S$ is an inverse semigroup and has
subgroups $G_{a}$ corresponding to loops which start and end at the
same point $a\in K$.

Let $K$ be a partially ordered set with binary relation $\leq$
satisfying reflexivity, anti\-symmetry and transitivity conditions.
We call the set $K$ a \emph{poset}. Elements $a$ and $b$ are called
\emph{comparable} on $K$ if $a\leq b$ or $b\leq a.$ We say that the
poset $K$ is \emph{upward directed} if for every pair $a,b\in K$
there exists $c\in K$, such that $a\leq c$ and $b\leq c.$

We call an ordered pair of comparable elements $a$ and $b$  on $K$
an \emph{elementary path}. We denote it by $(b,a)$ if $b\leq a$ and
by $\overline{(b,a)}$ if $b\geq a$ and say that $a$ is a
\emph{starting point} of $p$, and $b$ is an \emph{ending point}. We
use the notation $\partial_1p=a$ to denote the starting point of $p$
and $\partial_0p=b$ to denote the ending point. For an elementary
path $p=(b,a)$ we define the \emph{inverse path}
$p^{-1}=\overline{(a,b)}.$ For $p=\overline{(b,a)}$ the inverse path
is $p^{-1}=(a,b)$. Obviously, $(p^{-1})^{-1}=p.$ Finally we call the
pair $(a,a)=\overline{(a,a)}=i_a$ a \emph{trivial path}.

Let $p_1,\ldots,p_n$ be elementary paths, such that
$\partial_0p_{i-1}=\partial_1p_i$ for $i=2,\ldots,n.$ We define a
{\em path} $p$ as the sequence $$p=p_n\ast p_{n-1}\ast\ldots\ast
p_1.$$ The starting point of $p$ is $\partial_1p=\partial_1p_1$ and
the ending point is $\partial_0p=\partial_0p_n.$ For every path
$p=p_n\ast p_{n-1}\ast\ldots\ast p_1$ the inverse path is
$$p^{-1}=p_1^{-1}\ast p_{2}^{-1}\ast\ldots\ast p_n^{-1}$$ with
$\partial_1p^{-1}=\partial_0p$ and $\partial_0p^{-1}=\partial_1p.$
Let us consider the set of all paths on $K$. We denote an
\emph{empty path} $0$ as a formal symbol. The empty path $0$ has
neither the starting point nor the ending point. We define a
semigroup structure on the set of all paths with the empty path by
extending the operation "$\ast$" to multiplication as
$$
p\ast q=\left\{\begin{array}{ll}
       p\ast q & \mbox{if $p\neq 0$, $q\neq 0$ and $\partial_1p=\partial_0q$,}\\
       0 & \mbox{otherwise}
       \end{array}\right.
$$
for all paths $p$ and $q$.

The poset $K$ is called \emph{connected} if for all $a,b\in K$ there
exists a path $p$, such that $\partial_0p=a,$ $\partial_1p=b.$
Throughout the rest of this article we assume $K$ be a connected
set.

We call the set of all paths on $K$ with the empty path a \emph{path
semigroup} and denote it by $S$ if for all $a,b,c\in K$, such that
$a\leq b\leq c$, the following axioms hold:

1. $(a,b)\ast(b,c)=(a,c);$

2. $\overline{(c,b)}\ast\overline{(b,a)}=\overline{(c,a)};$

3. $\overline{(b,a)}\ast (a,b)=i_b,$ \hspace{0.2cm} $(a,b)\ast
\overline{(b,a)}=i_a;$

4. $(a,b)\ast i_b=(a,b),$ \hspace{0.2cm} $i_a\ast(a,b)=(a,b);$

5. $\overline{(b,a)}\ast i_a=\overline{(b,a)},$ \hspace{0.2cm}
$i_b\ast\overline{(b,a)}=\overline{(b,a)};$

6. $i_a\ast i_a=i_a.$

It is easy to see that path semigroup $S$ has the following useful
properties:

1) for every $p\in S,$ such that $\partial_0p=a$, $\partial_1p=b$,
$$p^{-1}\ast p=i_b,\hspace{0.2cm} p\ast p^{-1}=i_a;$$

2) for every $p\in S,$ such that $\partial_0p=a$, $\partial_1p=b$,
$$i_a\ast p=p\ast i_b=p;$$

3) for all $p,q\in S,$ such that $\partial_0q=\partial_1p$,
$$(p\ast q)^{-1}=q^{-1}\ast p^{-1};$$

4) for all $p,q,s\in S$ if $p\ast q=p\ast s\neq 0$ or $q\ast p=s\ast
p\neq 0$ then $q=s$; so the  path semigroup $S$ is a semigroup with
a reduction.

Thus, we can write elements of $S$ as follows:
\begin{equation}\label{path}
p=(a_{2n},a_{2n-1})\ast \ldots\ast \overline{(a_3,a_2)}\ast
(a_2,a_1)\ast \overline{(a_1,a_0)}.
\end{equation}
Here elementary paths of type $(a,b)$ and $\overline{(a,b)}$
alternate with each other. Note that there exists a variety of
representations of type (\ref{path}) for a path $p$. Our definition
of the path turns out to be in full accordance with the definition
given in \cite{RuzziVasselli}. The multiplication $(a_{i+1},a_i)\ast
\overline{(a_i,a_{i-1})}$ is 1-\emph{simplex} with support $a_i$
where elements $a_{i-1}, a_i, a_{i+1}$ are 0-\emph{simplices} (see
definitions of 0-simplex and 1-simplex in
\cite{Ruzzi,Roberts,RuzziVasselli}).

Three elements $a,c,x\in K$, such that $a,c\leq x$, form a 1-simplex
denoted by
$$[a^xc]=(a,x)\ast\overline{(x,c)}$$ with support $x$.
An inverse 1-simplex is
$$[c^xa]=(c,x)\ast \overline{(x,a)}$$ with the same
support. In general a 1-simplex depends on the support. But for
example if $x,y\in K$ are comparable elements then
\begin{equation}\label{1sim} [a^xc]=[a^yc].\end{equation}
Indeed for $x\leq y$ we observe
$$[a^yc]=(a,y)\ast\overline{(y,c)}=(a,x)\ast(x,y)\ast\overline{(y,x)}\ast\overline{(x,c)}
=(a,x)\ast i_x\ast\overline{(x,c)}=[a^xc].$$ In Lemma
\ref{UpDirecSet} we show that a 1-simplex does not depend from the
support if the poset is upward directed.

Therefore, one can rewrite the path (\ref{path}) as a sequence of
1-simplices:
$$p=[a_{2n}\,^{a_{2n-1}}\hspace{0,1cm}a_{2n-2}]\ast \ldots\ast[a_2\,^{a_1}\hspace{0.1cm}a_0].$$

Let us recall the definition of an inverse semigroup (for details
see \cite{Clifford,Paterson,Vagner}). Let $S$ be a semigroup.
Elements $a,b\in S$ are called \emph{mutual inverses} if
$$a=aba, \ \ b=bab.$$ The semigroup $S$ is called an \emph{inverse
semigroup} if for every $a\in S$ there exists a unique inverse
element $b\in S$.

We use the following theorem in the proof of Lemma \ref{Sinv}.

\begin{theorem}[\cite{Vagner} ]\label{InvSem}
For a semigroup $S$ in which every element has an inverse,
uniqueness of inverses is equivalent to the requirement that all
idempotents in $S$ commute.
\end{theorem}

\begin{lemma}\label{Sinv}
The path semigroup $S$ is an inverse semigroup.
\end{lemma}

\begin{proof} Let $p\in S$ be a path with a starting point $\partial_1p=a$
and an ending point $\partial_0p=b.$ For every $p$ there is an
inverse path $p^{-1}$, such that $$p\ast p^{-1}\ast p=i_{b}\ast p=p,
\ \ p^{-1}\ast p\ast p^{-1}=i_{a}\ast p^{-1}=p^{-1}.$$ Hence, $p$
and $p^{-1}$ are mutual inverses elements. For every $a\in K$ we
have $i_a\ast i_a=i_a$ and $i_a\ast i_b=0$ for all $a\neq b$.
Therefore the set $\{i_a\}_{a\in K}$ forms a commutative
subsemigroup of idempotents in the path semigroup $S$. Hence, by
Theorem \ref{InvSem} the path semigroup $S$ is an inverse semigroup.
\end{proof}

\begin{lemma}\label{ElemDefPuth}
If for some 1-simplices $[a^xb]$ and $[b^yc]$ there exists $z\in K$,
such that $x,y\leq z$, then $[a^xb]\ast [b^yc]=[a^zc].$
\end{lemma}

\begin{proof} We have $$[a^xb]\ast [b^yc] = (a,x)\ast \overline{(x,b)}\ast
(b,y)\ast \overline{(y,c)}$$
$$=(a,x)\ast(x,z)\ast\overline{(z,x)}\ast\overline{(x,b)}\ast(b,y)\ast(y,z)\ast\overline{(z,y)}
\ast\overline{(y,c)}$$
$$=(a,z)\ast\overline{(z,b)}\ast(b,z)\ast\overline{(z,c)}=(a,z)\ast\overline{(z,c)}=[a^zc].
$$
\end{proof}

\begin{corollary}\label{deform}
If for some 1-simplices $[a^xb]$, $[b^yc]$ and $[a^zc]$ there exists
$w\in K$, such that $x,y,z\leq w$, then $[a^xb]\ast [b^yc]=[a^zc].$
\end{corollary}
\begin{proof} Using the Lemma \ref{ElemDefPuth} and the equality
(\ref{1sim}) we have $[a^xb]\ast [b^yc]=[a^wc]=[a^zc].$
\end{proof}

In the works \cite{Roberts,RuzziVasselli} there exists the notion of
an elementary deformation of a path. They say that a path admits an
{\em elementary deformation} if one can replace some section
$[a^xb]\ast[b^yc]$ of the path with $[a^zc]$ and vice versa. It is
possible in the conditions of the Corollary \ref{deform}. If we can
obtain a path $q\in S$ from some path $p\in S$ by a finite number of
elementary deformations then according to the Lemma
\ref{ElemDefPuth} and the Corollary \ref{deform} we have the
equality $q=p.$

We say that $p\in S$ is a \textit{loop} if $\partial _{0} p=\partial
_{1} p$.

Let us denote by $G_{a}$ the set of all loops that start and end in
the point $a$.

\begin{lemma}\label{GrLoop} The following statements hold:

{\rm 1)} the set $G_{a} $ is a subgroup in $S$ with a unit $i_{a} $;

{\rm 2)} each path $p$ generates isomorphism between groups $G_{a}$
and $G_{b}$  if \ $\partial _{0} p=a,\; \;
\partial _{1} p=b$;

{\rm 3)} if $p,q\in S$ and $\partial _{0} p=\partial _{0} q=a,\; \;
\partial _{1} p=\partial _{1} q=b$, then there exist $g_{1} \in
G_{a} $ and $g_{2} \in G_{b} $, such that $p=g_{1} *q=q*g_{2} $.
\end{lemma}

\begin{proof} 1) The first statement is obvious.

2) Define a map $\gamma _{p} :G_{a} \to G_{b} $ in the following
way:
\[\gamma _{p} (g)=p^{-1} gp,\]
where $g\in G_{a} $. One can check that  $\gamma _{p} $ is an
isomorphism.

3) It is easy to see that the statement holds for $g_{1} =p*q^{-1}
\in G_{a} $ and $g_{2} =q^{-1} *p\in G_{b} $.
\end{proof}

\begin{lemma}\label{UpDirecSet}
If the poset $K$ is an upward directed set then the following
statements hold:

{\rm 1)} for all $a,b,x,y\in K$ if  $a,b\le x$ and $a,b\le y$ then
$$[a^xb]=(a,x)*\overline{(x,b)}=(a,y)*\overline{(y,b)}=[a^yb];$$
for simplicity let us omit supports and denote a 1-simplex by
$[a,b];$

{\rm 2)} $[a,b]*[b,c]=[a,c]$ for all  $a,b,c\in K;$

{\rm 3)} for every $p\in S$ if $\partial _{0} p=a$ and $\partial
_{1} p=b$ then $p=[a,b];$

{\rm 4)} if $g\in G_{a} $ then $g=i_{a} $ and the group $G_{a} $ is
a trivial group.
\end{lemma}

\begin{proof} 1) As the poset $K$ is upward directed set then there exists $z\in
K$, such that $x,y\le z$. Hence, we have
\begin{multline*}
[a^xb]=(a,x)*\overline{(x,b)}=(a,x)*(x,z)*\overline{(z,x)}*\overline{(x,b)}
=(a,z)*\overline{(z,b)}=\\
=(a,y)*(y,z)*\overline{(z,y)}*\overline{(y,b)}
=(a,y)*\overline{(y,b)}=[a^yb].\end{multline*}

2) It follows from Lemma \ref{ElemDefPuth}.

3) It follows from 2).

4) For every $g\in G_{a} $ we have $g=[a,a_{n} ]*\ldots*[a_{2},a_{1}
]*[a_{1},a]$. Using 2) several times, one gets $g=[a,a_{1}
]*[a_{1},a]=[a,a]=(a,a)*\overline{(a,a)}=i_{a}.$
\end{proof}

\section{$C^*$-algebra $C_{red}^{*}(S)$}

In this section we define the $C^*$-algebra $C_{red}^{*}(S)$
generated by the representation of the path semigroup $S$ and obtain
the net of isomorphic $C^*$-algebras
$(\mathcal{A}_a,\gamma_{ba},a\leq b)_{a,b\in K}$ over the poset $K$,
where $\gamma_{ba}:\mathcal{A}_a\rightarrow\mathcal{A}_b$ are
*-isomorphisms satisfying the identity
$\gamma_{cb}\circ\gamma_{ba}=\gamma_{ca}$ for $a\leq b\leq c.$

Let us consider a Hilbert space
\[l^{2} (S)=\left\{f:S\to \mathbb{C}\; \; |\; \; \sum _{p\in S}|f(p)|^{2}  <\infty \right\}\]
with an inner product $\left\langle f,g\right\rangle =\sum _{p\in
S}f(p)\overline{g(p)} $. A family of functions  $\left\{e_{p}
\right\}_{p\in S} $ is an ortonormal basis of $l^2(S)$ where $e_{p}
(p')=\delta _{p,p'} $  is a Kronecker symbol. Let  $B(l^{2} (S))$ be
the algebra of all linear bounded operators acting on $l^{2} (S)$.

Define a representation $\pi: S\rightarrow B(l^2(S))$ by
$\pi(p)=T_p$ where
$$T_pe_q=\left\{
 \begin{array}{ll}
 e_{p\ast q} & \mbox{if $\partial_1p=\partial_0q$,}
 \\
 0 & \mbox{otherwise.}\\ \end{array}\right.$$
Note that $\pi$ is the left regular representation and coincides
with the Vagner representation of an inverse semigroup (see the
definition of the Vagner representation in \cite{Paterson}).

We have $\left\langle T_{p} e_{q} ,e_{r} \right\rangle \ne 0$ if and
only if $p*q=r$ or $q=p^{-1} *r$. Hence,
\[\left\langle T_{p} e_{q} ,e_{r} \right\rangle =\left\langle e_{q} ,T_{p^{-1} } e_{r}
\right\rangle .\]

Define the adjoint operator $T_{p}^{*} =T_{p^{-1} } $. In Lemma
\ref{Tp} we show that operators $T_p$ and $T_p^*$ are partial
isometric operators.

Given $a\in K$ we define $S_{a} =\left\{p\in S\; \; |\; \; \partial
_{0} p=a\right\}$. Thus $l^2(S)$ can be written as
\[l^{2} (S)=\mathop{\oplus }\limits_{a\in K} l^{2} (S_{a} ).\]

\begin{lemma}\label{Tp} The following statements hold:

{\rm 1)} for every $p\in S$, such that $\partial _{0} p=a,\;
\partial _{1} p=b$, the operator  $T_{p} $ is a mapping from
 $l^{2} (S_{b} )$ to $l^{2} (S_{a} )$ and the operator $T_{p}^{*} $ is an inverse mapping from
 $l^{2} (S_{a} )$ to $l^{2} (S_{b})$;

 {\rm 2)} for every $p\in S$, such that $\partial _{0} p=a,\; \partial _{1} p=b$, operators
$I_{a} =T_{p} T_{p}^{*} $  and  $I_{b} =T_{p}^{*} T_{p} $ are
projectors on $l^{2} (S_{a} )$ and $l^{2} (S_{b} )$ respectively;

 {\rm 3)} for every $g\in G_{a} $ the operator $T_{g} $ is a unitary operator on $l^{2} (S_{a})$;

 {\rm 4)} for all $p,q\in S$, such  that $\partial _{0} p=\partial _{0} q=a,\; \partial _{1} p=\partial _{1}
 q=b$,
 there exist $g_{1} \in G_{a} $ and $g_{2} \in G_{b} $, such that $T_{p} =T_{g_{1} } T_{q} =T_{q} T_{g_{2}}$.
\end{lemma}

\begin{proof} 1) We observe that $T_{p} e_{q} =e_{p\ast q}$ if $\partial _{0} q=b$ and $T_pe_q=0$ otherwise.
Since $\partial _{0}
 (p*q)=a$ then
$T_{p} :l^{2} (S_{b} )\to l^{2} (S_{a} )$. Similarly, $T_{p}^{*}
:l^{2} (S_{a} )\to l^{2} (S_{b} )$.

 2) It is easy to see that $I_{a} e_{q}=T_{p} T_{p}^{*}e_q=e_{p\ast p^{-1}\ast q} =e_{q} $
 if $\partial _{0} q=a$ and $I_{a} e_{q} =0$ otherwise.
Therefore, $I_{a} $ is a projector on $l^{2} (S_{a} )$. Similarly,
one can prove that $I_{b} $ is a projector on $l^{2} (S_{b} )$.

 3) We have $T_{g} :l^{2} (S_{a} )\to l^{2} (S_{a} )$ and $T_{g} T_{g}^{*} e_{p}=e_{g\ast g^{-1}\ast p} =e_{p},$
  $T_{g}^{*} T_{g} e_{p} =e_{p} $
 for every $p\in S_{a} $. Hence, $T_g$ is a unitary operator.

 4) This statement follows from the Lemma \ref{GrLoop} (item 3).
\end{proof}

 Let us denote by $C_{red}^{*} (S)$ a uniformly closed subalgebra of $B(l^{2} (S))$ generated by operators
 $T_{p} $ for every $p\in S$.
 Obviously the set of finite linear combinations of operators $T_{p} $, $p\in S$,
 is dense in â $C_{red}^{*} (S)$.

Given $a\in K$ we denote $S^{a} =\left\{p\in S\; \; |\; \; \partial
_{1} p=a\right\}$. Thus we have again
\[l^{2} (S)=\mathop{\oplus }\limits_{a\in K} l^{2} (S^{a} ).\]

\begin{theorem}\label{PrSum} The following statements hold:

{\rm 1)} the algebra $C_{red}^{*} (S)$ is irreducible on $l^{2}
(S^{a} )$ for every $a\in K$;

{\rm 2)} $C_{red}^{*} (S)=\mathop{\oplus }\limits_{a\in K}
C_{red}^{*} (S)|_{l^{2} (S^{a} )} $ and every operator $A\in
C_{red}^{*} (S)$ can be represented as $A=\mathop{\oplus
}\limits_{a\in K} A_{a}$
 where $A_{a} =A|_{l^{2} (S^{a})}$;

{\rm 3)} if the group $G_{a} $ is non-trivial then $C_{red}^{*}
(S)|_{l^{2} (S^{a} )} $ doesn't contain compact operators.
\end{theorem}

\begin{proof} 1) The set $\{e_p, \partial_1p=a\}_{p\in S}$ is a basis
of
$l^{2} (S^{a} )$. For all  $p_{1} ,p_{2} \in S^{a} $ and $p=p_{2}
*p_{1}^{-1} $ we have $T_{p} e_{p_{1} } =e_{p_{2} } $.  It means
that the algebra $C_{red}^{*} (S)$ is irreducible on $l^{2} (S^{a}
)$.

2) This statement follows from the fact that for every $p\in S$ operator  $T_{p}
$ maps the space
 $l^{2} (S^{a} )$ onto itself for every $a\in K$.

3) Let $p\in S^{a} $, $g\in G_{a} $ and $g\ne i_{a} $. Consider the
sequence $x_{n} =e_{p*g^n} $ where $g^n=\underbrace{g\ast
g\ast\ldots\ast g}_n.$ Since $g*g\ne g$ elements of the sequence $\{
x_{n} \} $ are pairwise orthogonal. If $A\in C_{red}^{*} (S)|_{l^{2}
(S^{a} )}$ is a compact operator then $\left\| Ax_{n} \right\| \to
0$ as $n\to \infty $. On the other hand $Ae_{p}
=\sum\limits_{i}\alpha _{i} e_{p_{i} }  $ where $p_{i} \in S^{a}$
and $\alpha_i$ are complex coefficients. Referral to the fact that
$A$ is approximated by finite linear combinations of operators
 $T_{q} $,
$q\in S$, and to the equality $T_{q} e_{p*g} =e_{q*p*g} $ we
obtain $Ae_{p*g} =\sum\limits_{i}\alpha _{i} e_{p_{i} *g}  $.
Similarly $Ae_{p\ast g^n} =\sum\limits_{i}\alpha _{i} e_{p_{i}
\ast g^n} $ for all $n$. Therefore, for every
 $n$ we have $\left\| Ax_{n} \right\| =\left(\sum\limits
_{i}\left|\alpha _{i} \right|^{2}  \right)^{1/2} >0$. Hence, $A$ is
not a compact operator.
\end{proof}

\begin{theorem} Let $K$ be an upward directed set. Then the following statements hold:

{\rm 1)} for every $p\in S$, such that $\partial _{0} p=a,\;
\partial _{1} p=b$, we have  $T_{p} =T_{[a,b]}$;

{\rm 2)} for every $a\in K$ the algebra $C_{red}^{*} (S)|_{l^{2}
(S^{a} )} $ coincides with the algebra of all compact operators
 on $B(l^{2} (S^{a} ))$;

{\rm 3)} the algebra $C_{red}^{*} (S)$ is non-unital.
\end{theorem}

\begin{proof} 1) This statement follows from the Lemma \ref{UpDirecSet}.

2) The set $\{e_{[c,a]}\}_{c\in K}$ is a basis of $l^{2} (S^{a} )$.
For every operator $T_{p} $ we have $T_{p} e_{[c,a]} \ne 0$ if and
only if $\partial _{1} p=c$. Hence, $T_p = T_{[b,c]}$ for some $b$
and $T_{[b,c]}e_{[c,a]} = e_{[b,a]}$. Therefore, $T_{p} |_{l^{2}
(S^{a} )} $ is a one dimensional operator. So $C^*$-algebra
$C^*_{red}(S)|_{l^2(S^a)}$ coincides with the algebra of all compact
operators on $B(l^{2} (S^{a} ))$.

3) By the Theorem \ref{PrSum} for every element $A\in C_{red}^{*}
(S)$ we have $A=\mathop{\oplus }\limits_{a\in K} A_{a} $ where
$A_{a} \in C_{red}^{*} (S)|_{l^{2} (S^{a} )} $. If the algebra
$C_{red}^{*} (S)$ has the unit $I$ then $I_{a} =I|_{l^{2} (S^{a} )}
$ is a compact operator in the infinite dimensional Hilbert space.
This is a contradiction.
\end{proof}

Given $a\in K$ we denote $\mathcal{A}_a=C_{red}^{*} (S)|_{l^{2} (S^{a} )}$.

\begin{theorem} There exists the
set of *-isomorphisms
 $\{\gamma_{ba}, a\leq b\}_{a,b\in K}$:
\[\gamma_{ba}:\mathcal{A}_a \rightarrow \mathcal{A}_b,
\]
such that $\gamma_{cb}\circ\gamma_{ba}=\gamma_{ca}$ for all $a,b,c\in
K$ and $a\leq b\leq c$. And we obtain a net of isomorphic
$C^*$-algebras $\left\{\mathcal{A}_a,\gamma_{ba},a\leq b
\right\}_{a,b\in K}$ over the poset $K$.
\end{theorem}

\begin{proof} Define a unitary operator $U_{ab} :l^{2} (S^{a} )\to l^{2}
(S^{b} )$ for all $a,b\in K$, $a\leq b$, by
\[U_{ab} e_{q} =e_{q*(a,b)} \] for every
$q\in S^{a}$. Then  $U_{ab}^{*} =U_{\overline{ba}} :l^{2} (S^{b}
)\to l^{2} (S^{a} )$ is the adjoint operator. Obviously, $U_{ab}^{*}
U_{ab} =id|_{l^{2} (S^{a} )}$ and $U_{ab} U_{ab}^{*} =id|_{l^{2}
(S^{b} )}$. Let us define a mapping $\gamma _{ba} :\mathcal{A}_a
\rightarrow \mathcal{A}_b $ by
\[\gamma _{ba} (A)=U_{ab} AU_{ab}^{*}\]
for every $A\in\mathcal{A}_a$. One can check that $\gamma _{ba} $ is
the *-isomorphism. It remains to check the equality
$\gamma_{cb}\circ\gamma_{ba}=\gamma_{ca}$ for $a\leq b\leq c.$ We
observe that
$$(\gamma_{cb}\circ\gamma_{ba})(A)=\gamma_{cb}(\gamma_{ba}(A))=U_{bc}U_{ab}AU^*_{ab}U^*_{bc}$$
for every $A\in \mathcal{A}_a.$ Otherwise
$$U_{bc}U_{ab}e_q=U_{bc}e_{q\ast(a,b)}=e_{q\ast(a,b)\ast(b,c)}=e_{q\ast(a,c)}=U_{ac}
e_q$$ for every $q\in S^a$ and similarly
$U^*_{ab}U^*_{bc}e_p=U^*_{ac}e_p$ for every $p\in S^c.$ So
$(\gamma_{cb}\circ\gamma_{ba})(A)=\gamma_{ca}(A)$ for every $A\in
\mathcal{A}_a.$
\end{proof}

\begin{remark} The set of isomorphisms $\{\gamma_{ba}, a\leq b\}_{a,b\in K}$
can be extended from elementary paths to 1-simplices
$\{\gamma_{[b^xa]}, a,b\leq x\}_{a,b,x\in K}$ by
$\gamma_{[b^xa]}=\gamma^{-1}_{xb}\circ\gamma_{xa}$,  so that they
satisfy 1-cocycle identity \cite{RuzziVasselli}:
$\gamma_{[c^yb]}\circ\gamma_{[b^xa]}=\gamma_{[c^za]}$ for
$[c^yb]\ast[b^xa]=[c^za]$. Extending the set $\{\gamma_{[b^xa]},
a,b\leq x\}_{a,b,x\in K}$ to paths we get the set of isomorphisms
$\{\gamma_p\}_{p\in S}$ satisfying the equality
$\gamma_{p_2}\circ\gamma_{p_1}=\gamma_{p_2\ast p_1}$ for all
$p_1,p_2\in S$ and $\partial_0p_1=\partial_1p_2$.
\end{remark}

\section{Extensions of the $C^*$-algebra $C_{red}^{*} (S)$}

In this section we consider the extensions of the algebra
$C^*_{red}(S)$, such that this algebra is an ideal in that
extensions and quotient algebras are isomorphic to the Cuntz
algebra.

Let $K$ be an upward directed countable set. By the lemma
\ref{UpDirecSet} for every path  $p\in S$, such that $\partial _{0}
p=a,\;\partial _{1} p=b$, we have $p=[a,b]$. Let us represent the
set $K$ as a finite union of countable disjoint sets
$$K=\bigcup _{i=1}^{n}E_{i},  $$   where $E_{i} \bigcap
E_{j} =\emptyset$ for $i\ne j$.

We define one-to-one mappings $\phi _{i}:E_{i} \to K,$
$i=1,\ldots,n$, and operators $T_{\phi _{i}}: l^2(S)\to l^2(S)$ in
the following way:
\[T_{\phi _{i} } =\bigoplus\limits_{x\in E_{i} } T_{[x,\phi _{i} (x)]} , \  i=1,\ldots,n.\]
An adjoint operator of the operator $T_{\phi _{i} } $ is $$T_{\phi
_{i} }^{*} =\bigoplus\limits_{x\in E_{i} } T_{[x,\phi _{i} (x)]}^{*}
=\bigoplus\limits_{x\in E_{i} } T_{[\phi _{i}(x),
x]}=\bigoplus\limits_{x\in K} T_{[x,\phi _{i}^{-1} (x)]} $$.

The following equalities hold:
$$T_{\phi _{i} }^{*} T_{\phi _{i} } =id; \ \
T_{\phi _{i} }^{*} T_{\phi _{j} } =0, \ i\ne j; \ \ \sum
_{i=1}^{n}T_{\phi _{i} } T_{\phi _{i} }^{*}  =id.$$

Indeed every basis element has a form $e_{[a,b]} $. Therefore,
$$
{T_{\phi _{i} }^{*} T_{\phi _{i} } e_{[a,b]} = T_{\phi _{i} }^{*}
T_{[\phi _{i}^{-1} (a),a]} e_{[a,b]} = T_{\phi _{i} }^{*} e_{[\phi
_{i}^{-1} (a),b]} =}$$
$${=T_{[a,\phi _{i}^{-1} (a)]} e_{[\phi
_{i}^{-1} (a),b]} =e_{[a,b]} .} $$ Analogously, since $E_{i} \bigcap
E_{j} =\emptyset$ we have $T_{\phi _{i} }^{*} T_{\phi _{j} }
e_{[a,b]} =0.$ Finally if $a\in E_{k} $ then
$${\left(\sum _{i=1}^{n}T_{\phi _{i} }
T_{\phi _{i} }^{*} \right)e_{[a,b]} =T_{\phi _{k} } T_{[\phi _{k}
(a),a]} e_{[a,b]} =T_{\phi _{k} } e_{[\phi _{k} (a),b]} =}$$
$${=T_{[a,\phi _{k} (a)]} e_{[\phi _{k} (a),b]} =e_{[a,b]} .}
$$

Let us consider a uniformly closed subalgebra of $B(l^{2} (S))$
 generated by operators $T_{p} $, $p\in S$, and $T_{\phi _{i} } $, $i=1,\ldots,n$. Denote it by $C_{red,n}^{*} (S)$.
 The algebra $C_{red,n}^{*} (S)$ is unital. Hence, it doesn't coincide with $C_{red}^{*} (S)$.
 It is an extension of algebra $C_{red}^{*} (S)$. Moreover the following lemma holds.

\begin{lemma}\label{ideal} The algebra $C_{red}^{*} (S)$ is an ideal in $C_{red,n}^{*}
(S)$.\end{lemma}

\begin{proof} We have $T_{\phi _{i} } T_{[a,b]}
=T_{[x,b]} $ for some $x\in K$
 and $T_{[a,b]} T_{\phi _{i} } =T_{[a,y]} $ for some $y\in K$. Since every  element $A\in C_{red}^{*} (S)$
 can be approximated by finite linear combinations of operators $T_{[a,b]} $ then $T_{\phi _{i} } A$ and
 $AT_{\phi _{i} } \in C_{red}^{*} (S)$.  \end{proof}

Let us recall the definition of the Cuntz algebra. The
\textit{finite Cuntz algebra} $O_{n}$ is a $C^*$-algebra generated
by isometries $s_{1} ,\ldots,s_{n} $
 satisfying to the following conditions:
\[s_{i}^{*} s_{j} =\delta _{ij} id,\quad \sum _{i=1}^{n}s_{i} s_{i}^{*}  =id.\]

The \textit{infinite Cuntz algebra} $O_{\infty } $ is a
$C^*$-algebra generated by $s_{1} ,s_{2} ,\ldots$ and relations
$$s_{i}^{*} s_{j} =\delta _{ij} id,\quad \sum _{i=1}^{n}s_{i}
s_{i}^{*}  \le id$$ for every $n\in \mathbb{N}$.

\begin{theorem} There
exist an isomorphism $C_{red,n}^{*} (S)/C_{red}^{*} (S)\cong O_{n} $
and  a short exact sequence
\[0\to C_{red}^{*} (S)\mathop{\to }\limits^{id} C_{red,n}^{*} (S)\mathop{\to }\limits^{\pi } O_{n} \to 0,\]
where $id$ is an embedding map and $\pi $ is a quotient map.
\end{theorem}

\begin{proof} Equivalence classes $[T_{\phi _{i} } ]=
 T_{\phi _{i} } +C_{red}^{*} (S)$, $i=1,\ldots,n$,
 are generators of the quotient algebra $C_{red,n}^{*} (S)/C_{red}^{*} (S)$.
 These classes are isometric operators satisfying the following identity:
\[\sum _{i=1}^{n}[T_{\phi _{i} } ][T_{\phi _{i} }^{*} ] =id.\]
Due to the universality of the Cuntz algebra we observe that
$$C_{red,n}^{*} (S)/C_{red}^{*} (S)\cong O_{n}. $$
\end{proof}

Now let us represent the set $K$ as a countable union of disjoint
countable sets:
\[K=\bigcup _{i=1}^{\infty }E_{i} \]
and define operators $T_{\phi _{i}}: l^2(S)\to l^2(S)$ in the
following way:
\[T_{\phi _{i} } =\bigoplus\limits_{x\in E_{i} } T_{[x,\phi _{i} (x)]} , \  i=1,2,\ldots.\]
By applying the reasoning used above one can prove the following
equalities:
$$T_{\phi _{i} }^{*} T_{\phi _{i} } =id; \ \
T_{\phi _{i} }^{*} T_{\phi _{j} } =0, \ i\ne j; \ \ \sum
_{i=1}^{n}T_{\phi _{i} } T_{\phi _{i} }^{*}  \leq id $$ for every
$n\in \mathbb{N}$.

Let us denote by $C_{red,\infty }^{*} (S)$ the uniformly closed
subalgebra of $B(l^{2} (S))$ generated by operators $T_{p} $, $p\in
S$, and $T_{\phi _{i} } $, $i=1,2,\ldots$.

Similarly to the Lemma \ref{ideal} the algebra $C_{red}^{*} (S)$ is
an ideal in $C_{red,\infty}^{*} (S)$ and for the infinite Cuntz
algebra the following theorem  holds.

\begin{theorem}
There exist an isomorphism $C_{red,\infty}^{*} (S)/C_{red}^{*}
(S)\cong O_{\infty } $ and  a short exact sequence
\[0\to C_{red}^{*} (S)\mathop{\to }\limits^{id} C_{red,\infty}^{*} (S)\mathop{\to }\limits^{\pi } O_{\infty} \to 0,\]
where $id$ is an embedding map and $\pi $ is a quotient map.
\end{theorem}

\textbf{Acknowledgements.} We thank E. Vasselli for helpful comments which have led to significant improvements.


\begin{thebibliography}{99}

\bibitem{Haag}  R.~Haag,  \emph{Local quantum Physics: Fields, particles,
algebras} (Springer-Verlag, Berlin, 1992).

\bibitem{Ruzzi}  G.~Ruzzi, Homotopy of posets, net cohomology and
superselection sectors in globally hyperbolic space-times, Rev.
Math. Phys., \textbf{17}, 1021-1070 (2005).

\bibitem{Roberts} J.E.~Roberts, More lectures in algebraic quantum field
theory, in: S. Doplicher, R. Longo (Eds.), Noncommutative Geometry.
C.I.M.E. Lectures, Martina Franca, Italy, 2000, Springer-Verlag
(2003).


\bibitem{RuzziVasselli}  G.~Ruzzi, E.~Vasselli, A new light on nets of C*-algebras
and their representations, Comm. Math. Phys., \textbf{312}, 655-694
(2012).

\bibitem{Vasselli1}  E.~Vasselli, Presheaves of symmetric tensor categories and
nets of C*-algebras, Journal of Noncommutative Geometry, \textbf{9},
121-159 (2015).

\bibitem{Vasselli2} E.~Vasselli, Presheaves of superselection structures in
curved spacetimes, Comm. Math. Phys., \textbf{335}, 895-933 (2015).

\bibitem{BrunettiRuzzi}  R.~Brunetti, G.~Ruzzi, Quantum charges and space-time
topology: The emergence of new superselection sectors, Comm. Math.
Phys., \textbf{287}, 523-563 (2009).

\bibitem{AGL}  M.A.~Aukhadiev,  S.A.~Grigoryan, E.V.~Lipacheva, Infinite-dimensional
compact quantum semigroup, Lobachevskii Journal of Mathematics,
\textbf{32} (4), 304-316 (2011).

\bibitem{AT} M.A.~Aukhadiev,  V.H.~Tepoyan, Isometric representations of totally
ordered semigroups, Lobachevskii  Journal of Mathematics,
\textbf{33} (3), 239-243 (2012).

\bibitem{GT}  S.A.~Grigoryan,  V.H.~Tepoyan, On isometric representations of the
perforated semigroup, Lobachevskii  Journal of Mathematics,
\textbf{34} (1), 85-88 (2013).

\bibitem{GLT} T.A.~Grigoryan, E.V.~Lipacheva,  V.H.~Tepoyan, On the extension of the
Toeplitz algebra, Lobachevskii  Journal of Mathematics, \textbf{34}
(4), 377-383 (2013).

\bibitem{ANS} M.A.~Aukhadiev,  A.S.~Nikitin,
A.S.~Sitdikov,  Crossed product of the canonical anticommutative
relations algebra in the Cuntz algebra, Russian Mathematics,
\textbf{58} (8), 71-73 (2014).

\bibitem{Clifford} A.H.~Clifford, G.B.~Preston, \emph{The algebraic theory of
semigroups} V.1. (AMS, 1964).

\bibitem{Paterson} Alan~L.T.~Paterson, \emph{Groupoids, Inverse Semigroups, and their
Operator Algebras} (Birkhauser, 1998).

\bibitem{Vagner} V.V.~Vagner, Generalized groups (Russian), Doklady Akad.
Nauk SSSR \textbf{84}, 1119-1122 (1952).


\end{thebibliography}
\end{document}